\documentclass{proc-l}

\newtheorem{theorem}{Theorem}[section]
\newtheorem{lemma}[theorem]{Lemma}
\newtheorem{corollary}[theorem]{Corollary}
\newtheorem{proposition}[theorem]{Proposition}

\theoremstyle{definition}
\newtheorem{definition}[theorem]{Definition}

\theoremstyle{remark}
\newtheorem{remark}[theorem]{Remark}

\numberwithin{equation}{section}

\begin{document}

\title{Discrete fractional calculus and the Saalschutz theorem}

\author{Rui A. C. Ferreira}
\address{Grupo F\'isica-Matem\'atica, Faculdade de Ci\^encias, Universidade de Lisboa,
Av. Prof. Gama Pinto, 2, 1649-003, Lisboa, Portugal.}
\email{raferreira@fc.ul.pt}
\thanks{Rui A. C. Ferreira was supported by the ``Funda\c{c}\~{a}o para a Ci\^encia e a Tecnologia (FCT)" through the program ``Stimulus of Scientific Employment, Individual Support-2017 Call" with reference CEECIND/00640/2017.}


\date{}


\begin{abstract}
In this work we present a novel proof of the Saalschutz formula by using the theory of discrete fractional calculus. The proofs of some results within this theory, namely, the fractional power rule and the fractional Leibniz rule are revisited.
\end{abstract}

\maketitle

\section{Introduction}

This work had its origin when the author found a (presumably) novel way to obtain the  Saalschutz formula (cf. \cite[pag. 49]{Slater}), namely,

\begin{equation}\label{salsc0}
\frac{(c-a)_m(c-b)_m}{(c)_m(c-a-b)_m}={_3F}_2(a,b,-m;c,1+a+b-c-m;1),
\end{equation}
using the discrete fractional calculus theory. In order to understand \eqref{salsc0}, we introduce the following concepts:

\begin{definition}\label{def1}
The Pochhammer symbol is defined, for $x,y\in\mathbb{R}$, by
$$(x)_y=\left\{\begin{array}{llll}x(x+1)\ldots(x+y-1)\ \mbox{for}\ y\in\mathbb{N}_1, \\
1\ \mbox{for } y=0,\\
\frac{\Gamma(x+y)}{\Gamma(x)}\ \mbox{for }x,x+y\notin\mathbb{N}^0\\
0\ \mbox{for }x\in\mathbb{N}^0\mbox{ and }x+y\notin\mathbb{N}^0.
    \end{array}\right.
$$
\end{definition}

\begin{definition}
The function ${_3F_2}$ above, known as a hypergeometric function, is defined by
$${_3F_2}(a_1,a_2,a_3;b_1,b_2;z)=\sum_{k=0}^\infty\frac{(a_1)_k(a_2)_k(a_3)_k}{(b_1)_k(b_2)_k}\frac{z^n}{n!},$$
when the series converges.
\end{definition}
Meanwhile, while deriving \eqref{salsc0}, we recast some results already known in the literature, namely the \emph{fractional power rule} and the \emph{fractional Leibniz rule}. In particular, we will present novel proofs for the fractional power rule and for the Leibniz rule respectively, pointing out some inconsistencies that we detected in the process (cf. Remark \ref{rem000} and Remark \ref{rem0001} below).

In this work we use the Miller--Ross fractional sum-difference operator introduced in 1988 in \cite{MillerRoss}. We give its definition in the next section, more precisely, in Definition \ref{deffracsum}. For this operator, we present and prove fractional power rules as well as the fractional Leibniz rule. These are the key ingredients to obtain the equality \eqref{salsc0}, though with some restrictions in the parameters.

The theory of discrete fractional calculus and fractional difference equations has been developed in recent years in several directions and is currently a hot topic of research (cf. for example \cite{Atici2,Ferreira1,Ferreira3,GoodrichBook,Goodrich,Lizama}). This work aims at contributing to the development of the theory, namely, by showing how ancient results may be obtained from manipulating with (discrete) fractional calculus formulas.

\section{Definitions and preliminary results}

In this section we introduce some of the concepts used within the discrete fractional calculus (mainly for the ``delta" $\Delta$-case) theory as well as some basic facts about it.

For a number $a\in\mathbb{R}$ we put $\mathbb{N}_a=\{a,a+1,\ldots\}$ and $\mathbb{N}^a=\{\ldots,a-1,a\}$. Sometimes we will also write $\mathbb{N}_a^b=\{a,a+1,\ldots,b\}$, where $b=a+k$ with $k\in\mathbb{N}_1$. 

\begin{definition}
Consider a function $f:\mathbb{N}_a\to\mathbb{R}$.
The forward difference operator\index{Forward difference operator} is defined by $\Delta [f](t)=f(t+1)-f(t)$, for $t\in\mathbb{N}_a$. Also, we define higher order differences recursively as $\Delta^n [f] (t)=\Delta [\Delta^{n-1} f] (t)$, $n\in\mathbb{N}_1$, where $\Delta^0$ is the identity operator, i.e. $\Delta^0 f(t)=f(t)$.
\end{definition}

\begin{definition}\label{def1}
The falling function is defined, for $x,y\in\mathbb{R}$, by
$$(x)^{\underline{y}}=\left\{\begin{array}{llll}x(x-1)\ldots(x-y+1)\ \mbox{for}\ y\in\mathbb{N}_1, \\
1\ \mbox{for } y=0,\\
\frac{\Gamma(x+1)}{\Gamma(x+1-y)}\ \mbox{for }x,x-y\notin\mathbb{N}^{-1}\\
0\ \mbox{for }x\notin\mathbb{N}^{-1}\mbox{ and }x-y\in\mathbb{N}^{-1}.
    \end{array}\right.
$$
\end{definition}
When they are defined, the following formulas
\begin{equation}\label{eq1}
   t^{\underline{\alpha+\beta}}=(t-\beta)^{\underline{\alpha}}t^{\beta},
\end{equation}
$$\Delta[s^{\underline{\alpha}}](t)=\alpha t^{\underline{\alpha-1}},$$
and
\begin{equation*}
  (t+\alpha-1)^{\underline{\alpha}}=(t)_{\alpha},
\end{equation*}
hold, and will be widely used throughout this manuscript

We now introduce the Miller and Ross fractional operator:

\begin{definition}\label{deffracsum}
Let $a\in\mathbb{R}$, $\nu\in\mathbb{R}\backslash\mathbb{N}^{0}$ and $f:\mathbb{N}_a\to\mathbb{R}$. Then the fractional sum-difference operator of order $\nu$ is defined by
\begin{equation}\label{fractsumdiff}
    \Delta_a^{-\nu}f(t)=\frac{1}{\Gamma(\nu)}\sum_{s=a}^{t-\nu}(t-(s+1))^{\underline{\nu-1}}f(s),\quad t\in\mathbb{N}_{a+\nu}.
\end{equation}
\end{definition}
Miller and Ross called \eqref{fractsumdiff} the fractional sum of order $\nu$ if $\nu>0$. We use here the nomenclature ``sum-difference" to include the ``diferences", that is, when one is considering $\nu<0$.

It is worth mentioning that Atici and Eloe (who built on the work of Miller and Ross roughly 20 years later \cite{Atici1}) defined the fractional sum of order $\nu>0$ by \eqref{fractsumdiff} but defined the fractional difference of order $\mu>0$ by (in analogy with the Riemann--Liouville fractional derivative)
$$\Delta_a^\mu f(t)=\Delta^n[\Delta_a^{-(n-\mu)}f](t),\quad t\in\mathbb{N}_{a+n-\mu},$$
where $n\in\mathbb{N}_1$ is such that $n-1<\mu\leq n$. Let us briefly describe the difference between both concepts. To simplify, consider $0<\mu<1$. Then, the Miller--Ross fractional difference of order $\mu$ is
\begin{equation}\label{fractsumdiff11}
    {^{M-R}\Delta}_a^{\mu}f(t)=\frac{1}{\Gamma(-\mu)}\sum_{s=a}^{t+\mu}(t-(s+1))^{\underline{-\mu-1}}f(s),\quad t\in\mathbb{N}_{a-\mu},
\end{equation}
while the Atici--Eloe one is given by
\begin{equation}\label{A-E}
    {^{A-E}\Delta}_a^\mu f(t)=\frac{\Delta}{\Gamma(1-\mu)}\sum_{s=a}^{t+\mu-1}(t-(s+1))^{\underline{-\mu}}f(s),\quad t\in\mathbb{N}_{a+1-\mu}.
\end{equation}
Let us expand \eqref{A-E}. We have ($\Delta_t$ stands for the derivative with respect to $t$),
\begin{align*}
    {^{A-E}\Delta}_a^\mu f(t)&=\frac{1}{\Gamma(1-\mu)}\left(\sum_{s=a}^{t+\mu}(t-s)^{\underline{-\mu}}f(s)-\sum_{s=a}^{t+\mu-1}(t-s-1)^{\underline{-\mu}}f(s)\right)\\
    &=\frac{1}{\Gamma(1-\mu)}\left(\sum_{s=a}^{t+\mu-1}\Delta_t(t-s-1)^{\underline{-\mu}}f(s)+\Gamma(1-\mu)f(t+\mu)\right)\\
    &=\frac{1}{\Gamma(-\mu)}\sum_{s=a}^{t+\mu-1}(t-s-1)^{\underline{-\mu-1}}f(s)+f(t+\mu)\\
    &=\frac{1}{\Gamma(-\mu)}\sum_{s=a}^{t+\mu}(t-s-1)^{\underline{-\mu-1}}f(s),\quad t\in\mathbb{N}_{a+1-\mu}.
\end{align*}
Therefore, we may say that the Miller--Ross fractional difference and the Atici--Eloe fractional difference coincide but have different domains of definition.

\section{Main results}

We start with a binomial-type formula for which, for completeness, we present a proof.

\begin{lemma}[Discrete analogue of the Binomial Theorem]
Let $x,y\in\mathbb{R}$ and $n\in\mathbb{N}_0$. Then
\begin{equation}\label{eq0}
    (x+y)^{\underline{n}}=\sum_{k=0}^{n}\binom{n}{k}x^{\underline{n-k}}y^{\underline{k}}.
\end{equation}
\end{lemma}

\begin{proof}
We use induction on $n$. The case $n=0$ is obvious. Now, suppose \eqref{eq0} holds for any real numbers $x,y$. Then,
\begin{align*}
    &\sum_{k=0}^{n+1}\binom{n+1}{k}x^{\underline{n+1-k}}y^{\underline{k}}=x^{\underline{n+1}}+y^{\underline{n+1}}
    +\sum_{k=1}^{n}\left[\binom{n}{k}+\binom{n}{k-1}\right]x^{\underline{n+1-k}}y^{\underline{k}}\\
    &=\sum_{k=0}^{n}\binom{n}{k}x^{\underline{n+1-k}}y^{\underline{k}}+\sum_{k=0}^{n-1}\binom{n}{k}x^{\underline{n-k}}y^{\underline{k+1}}+y^{\underline{n+1}}\\
    &=\sum_{k=0}^{n}\binom{n}{k}x^{\underline{n-k}}(x-n+k)y^{\underline{k}}+\sum_{k=0}^{n}\binom{n}{k}x^{\underline{n-k}}y^{\underline{k}}(y-k)\\
   &=(x+y-n)(x+y)^{\underline{n}}=(x+y)^{\underline{n+1
    }},
\end{align*}
where we have used \eqref{eq1} repeatedly. The proof is done.
\end{proof}

\begin{remark}
It can also be shown that,
\begin{equation}\label{PochBinom}
    (x+y)_{n}=\sum_{k=0}^{n}\binom{n}{k}(x)_{n-k}(y)_{k},
\end{equation}
for all $x,y\in\mathbb{R}$ and $n\in\mathbb{N}_0$.
\end{remark}
We now proceed to present the fractional power rule. The proof is inspired by the work of Gray and Zhang \cite{Gray}.

\begin{theorem}\label{sumPR}
Let $a\in\mathbb{R}$. Assume $\mu\in\mathbb{R}\backslash\mathbb{N}^{-1}$ and $\nu\in\mathbb{R}\backslash\mathbb{N}^0$. Then
\begin{equation}\label{FPSS}
    \Delta_{a+\mu}^{-\nu}[(s-a)^{\underline{\mu}}](t)=\frac{\Gamma(\mu+1)}{(t-a-\mu-\nu)!}(\mu+\nu+1)_{t-a-\mu-\nu},\mbox{ for } t\in\mathbb{N}_{a+\mu+\nu}.
\end{equation}
\end{theorem}

\begin{proof}
Let $t\in\mathbb{N}_{a+\mu+\nu}$. Then
\begin{align*}
     \Delta_{a+\mu}^{-\nu}[(s-a)^{\underline{\mu}}](t)&=\frac{1}{\Gamma(\nu)}\sum_{s=a+\mu}^{t-\nu}(t-(s+1))^{\underline{\nu-1}}(s-a)^{\underline{\mu}}\\
     &=\sum_{s=a+\mu}^{t-\nu}\frac{1}{\Gamma(\nu)}\frac{\Gamma(t-s)}{\Gamma(t-s+1-\nu)}\frac{\Gamma(s-a+1)}{\Gamma(s-a+1-\mu)}\\
     &=\sum_{s=0}^{t-a-\mu-\nu}\frac{1}{\Gamma(\nu)}\frac{\Gamma(t-s-a-\mu)}{\Gamma(t-s-a-\mu+1-\nu)}\frac{\Gamma(s+\mu+1)}{\Gamma(s+1)}.
\end{align*}
Put $n=t-a-\mu-\nu\in\mathbb{N}_0$. Then we get from the previous equality,
\begin{align*}
     \Delta_{a+\mu}^{-\nu}[(s-a)^{\underline{\mu}}](t)
     &=\sum_{s=0}^{n}\frac{1}{\Gamma(\nu)}\frac{\Gamma(n-s+\nu)}{\Gamma(n-s+1)}\frac{\Gamma(s+\mu+1)}{\Gamma(s+1)}\\
     &=\frac{\Gamma(\mu+1)}{n!}\sum_{s=0}^{n}\frac{n!}{(n-s)!s!}\frac{\Gamma(n-s+\nu)}{\Gamma(\nu)}\frac{\Gamma(s+\mu+1)}{\Gamma(\mu+1)}\\
     &=\frac{\Gamma(\mu+1)}{n!}\sum_{s=0}^{n}\binom{n}{s}(\nu)_{n-s}(\mu+1)_s\\
     &=\frac{\Gamma(\mu+1)}{n!}(\mu+\nu+1)_n,
\end{align*}
where we have used \eqref{PochBinom}. Finally, substituting $n$ by $t-a-\mu-\nu$, we obtain \eqref{FPSS}.
\end{proof}

\begin{corollary}\label{Cor1}
Let $a\in\mathbb{R}$. Assume $\mu\in\mathbb{R}\backslash\mathbb{N}^{-1}$ and $\nu\in\mathbb{R}\backslash\mathbb{N}^0$.

If $\mu+\nu\notin\mathbb{N}^{-1}$, then
\begin{equation}\label{FPS11}
    \Delta_{a+\mu}^{-\nu}[(s-a)^{\underline{\mu}}](t)=\frac{\Gamma(\mu+1)}{\Gamma(\mu+\nu+1)}(t-a)^{\underline{\mu+\nu}},\mbox{ for } t\in\mathbb{N}_{a+\mu+\nu},
\end{equation}
while, if $\mu+\nu\in\mathbb{N}^{-1}$, then
\begin{equation}\label{FPS12}
\Delta_{a+\mu}^{-\nu}[(s-a)^{\underline{\mu}}](t)=0,\mbox{ for } t\in\mathbb{N}_{a}.
\end{equation}
\end{corollary}

\begin{proof}
If $\mu+\nu\notin\mathbb{N}^{-1}$, then 
$$(\mu+\nu+1)_{t-a-\mu-\nu}=\frac{\Gamma(t-a+1)}{\Gamma(\mu+\nu+1)},\quad t\in\mathbb{N}_{a+\mu+\nu}.$$
Hence, \eqref{FPSS} becomes
$$\Delta_{a+\mu}^{-\nu}[(s-a)^{\underline{\mu}}](t)=\frac{\Gamma(\mu+1)}{\Gamma(\mu+\nu+1)}\frac{\Gamma(t-a+1)}{\Gamma(t-a+1-\mu-\nu)},$$
which is just \eqref{FPS11}.

If $\mu+\nu\in\mathbb{N}^{-1}$ then, using Definition \ref{def1}, $(\mu+\nu+1)_{t-a-\mu-\nu}=0$ for $t\in\mathbb{N}_a$, and this concludes the proof.
\end{proof}

\begin{remark}
Note that, for the values of $t$ such that the function $f(t)=(t-a)^{\underline{\mu}}$ is well-defined, we have $\Delta^m f(t)=\mu^{\underline{m}}(t-a)^{\underline{\mu-m}}$, for $m\in\mathbb{N}_1$. Formally, this formula is in accordance with \eqref{FPS11}.
\end{remark}

\begin{remark}\label{rem000}
The formula \eqref{FPS11} is in many contexts presented assuming only that $\mu\in\mathbb{R}\backslash\mathbb{N}^{-1}$ (cf., e.g.,  \cite[Lemma 2.3.]{Atici1} and \cite[Lemma 3.1.]{Holm}). However, the authors usually present the falling function as
$$t^{\underline{\nu}}=\frac{\Gamma(t+1)}{\Gamma(t+1-\nu)},$$
for any $t,\ \nu\in\mathbb{R}$ for which the right-hand side is well-defined, which is rather vague. In Corollary \ref{Cor1} we clearly stated the set of values for the parameters $a,\mu,\nu$ in which equality \eqref{FPS11} is valid.
\end{remark}

\begin{remark}
Some interesting consequences may be extracted from Corollary \ref{Cor1}. Under the conditions of it, consider $a=0$, $\mu+\nu\in\mathbb{N}^{-1}$, and fix equality \eqref{FPS12}. We then have,
$$\frac{1}{\Gamma(\nu)}\sum_{k=\mu}^{t-\nu}\frac{\Gamma(t-k)}{\Gamma(t-k+1-\nu)}\frac{\Gamma(k+1)}{\Gamma(k+1-\mu)}=0,\quad t\in\mathbb{N}_0,$$
which is equivalent to
$$\sum_{k=0}^{t-(\mu+\nu)}\frac{\Gamma(t-(k+\mu))}{\Gamma(t-(k+\mu)+1-\nu)}\frac{\Gamma(k+\mu+1)}{\Gamma(k+\mu+1-\mu)}=0.$$
Let $n=t-(\mu+\nu)$ and note that $n\in\mathbb{N}_{-(\mu+\nu)}$. It follows from the previous equality, 
$$\sum_{k=0}^{n}\frac{\Gamma(n+\nu-k)}{(n-k)!}\frac{\Gamma(k+\mu+1)}{k!}=0,$$
or
$$\sum_{k=0}^{n}\binom{n}{k}\Gamma(n+\nu-k)\Gamma(k+\mu+1)=0,\ n\in\mathbb{N}_{-(\mu+\nu)}.$$
\end{remark}

\begin{remark}\label{rem0001}
A word of caution: in the context of discrete fractional calculus it is essential to keep track of the domains of the fractional operators involved or one may be lead to some inconsistencies.  As an example, we now consider \cite[Property 6]{Gray}. There, it is \emph{shown} that (with a slight change in notation)
$$\nabla_{a+1}^\alpha[(s-a)_p](t)=0.$$
Nothing is written about the domain of the nabla operator $\nabla^\alpha_{a+1}$ (cf. \eqref{nabla1} below). When proving \cite[Property 6--(i)]{Gray} the authors use the following definition for $\nabla^\alpha_{a+1}$ (which seems to derive from \cite[Property 1]{Gray}):
\begin{equation}\label{nabla1}
\nabla_{a+1}^\alpha[(s-a)_p](t)=\frac{1}{\Gamma(-\alpha)}\sum_{j=a+1}^t(t-j+1)_{-\alpha-1}(j-a)_p.
\end{equation}
From \cite[Definition 2]{Gray} one can only assume that $t\in\mathbb{N}_{a+1}$. However, if we consider \cite[Example 1]{Gray}, we get,
$$\nabla_{1}^{3/2}[(s)_{1/2}](1)=\frac{1}{\Gamma(-\frac{3}{2})}\sum_{j=1}^1(1-j+1)_{-\frac{3}{2}-1}(j)_{\frac{1}{2}}=\Gamma\left(\frac{3}{2}\right)$$
and not zero as stated therein.

We point out that we may derive a correct formula for the nabla case, i.e. for \eqref{nabla1}, by using \eqref{FPS12}. Indeed, for $\mu+\nu=-m$ with $m\in\mathbb{N}_1$ and $t\in\mathbb{N}_a$, we have
\begin{align*}
    &\sum_{s=a+\mu}^{t-\nu}(t-s-1)^{\underline{\nu-1}}(s-a)^{\underline{\mu}}=0\\
    \Leftrightarrow &\sum_{s=a}^{t-\mu-\nu}(t-s-\mu-1)^{\underline{\nu-1}}(s-a+\mu)^{\underline{\mu}}=0\\
     \Leftrightarrow &\sum_{s=a}^{t+m}(t-s+1+m)_{\nu-1}(s-a+1)_{\mu}=0\\
        \Leftrightarrow &\sum_{s=a+1}^{t+m+1}(t+m+1-s+1)_{\nu-1}(s-a)_{\mu}=0.
\end{align*}
Therefore,
$$\frac{1}{\Gamma(-\alpha)}\sum_{j=a+1}^{t}(t-j+1)_{-\alpha-1}(j-a)_{p}=0,\quad t\in\mathbb{N}_{a+1+m},\ \alpha-p=m\in\mathbb{N}_1.$$
To finalize this issue we test the previous formula with $a=0$, $\alpha=3/2$, $p=1/2$ (hence, $m=1$), and $t=2$:
$$\sum_{j=1}^{2}(2-j+1)_{-3/2-1}(j)_{1/2}=(2)_{-5/2}(1)_{1/2}+(1)_{-5/2}(2)_{1/2}=0.$$
\end{remark}
Now we need a preparation lemma for the Leibniz rule.

\begin{lemma}\label{lem3.18-1}
Suppose that $g:\mathbb{N}_a\to\mathbb{R}$ and $k\in\mathbb{N}_0$, $\alpha\in\mathbb{R}$. Then
$$\sum_{n=0}^{k}(-1)^n\binom{k}{n}\Delta^ng(t-\alpha-n)=g(t-\alpha-k),\quad t\in\mathbb{N}_{a+\alpha+k}.$$
\end{lemma}

\begin{proof}
We use induction on $k$. So, let $k=0$ and $t\in\mathbb{N}_{a+\alpha}$. Then the equality trivially holds. Now, assume the equality holds for $k\in\mathbb{N}_0$ and let $t\in\mathbb{N}_{a+\alpha+k+1}$. We have,
\begin{align*}
    &\sum_{n=0}^{k+1}(-1)^n\binom{k+1}{n}\Delta^ng(t-\alpha-n)\\
    &=g(t-\alpha)+(-1)^{k+1}\Delta^{k+1}g(t-\alpha-(k+1))\\
    &\hspace{4cm}+\sum_{n=1}^{k}(-1)^n\left[\binom{k}{n}+\binom{k}{n-1}\right]\Delta^ng(t-\alpha-n)\\
    &=\sum_{n=0}^{k}(-1)^n\binom{k}{n}\Delta^ng(t-\alpha-n)-\sum_{n=0}^{k}(-1)^n\binom{k}{n}\Delta^n\Delta g(t-1-\alpha-n)\\
    &=g(t-\alpha-k)-\Delta g(t-1-\alpha-k)\\
   &=g(t-\alpha-(k+1)),
\end{align*}
and the proof is done.
\end{proof}
The following result was proved for the first time in the seminal work by Miller and Ross \cite{MillerRoss} and later in \cite{Atici3}. In both works it is assumed that the function $g$ is defined in a union of discrete domains. Here, $g$ is defined on $\mathbb{N}_a$.

\begin{theorem}[Leibniz rule]\label{thmL-R}
Suppose that $f,g:\mathbb{N}_a\to\mathbb{R}$ and $\alpha\in\mathbb{R}\backslash\mathbb{N}^{0}$. Then,
\begin{equation}\label{FractLeiRule}
    \Delta_a^{-\alpha}[fg](t)=\sum_{n=0}^{t-\alpha-a}\binom{-\alpha}{n}\Delta_a^{-(\alpha+n)}f(t)\cdot\Delta^ng(t-\alpha-n),\quad t\in\mathbb{N}_{a+\alpha}.
\end{equation}
\end{theorem}

\begin{proof}
We start by fixing $t\in\mathbb{N}_{a+\alpha}$. By Lemma \ref{lem3.18-1}, we have
$$\sum_{n=0}^{t-s-\alpha}(-1)^n\binom{t-s-\alpha}{n}\Delta^ng(t-\alpha-n)=g(s),\quad s\in\mathbb{N}_{a}^{t-\alpha}.$$
Therefore,
\begin{align*}
     &\Delta_a^{-\alpha}[fg](t)\\
     &=\frac{1}{\Gamma(\alpha)}\sum_{s=a}^{t-\alpha}(t-(s+1))^{\underline{\alpha-1}}f(s)g(s)\\
     &=\frac{1}{\Gamma(\alpha)}\sum_{s=a}^{t-\alpha}\sum_{n=0}^{t-s-\alpha}(t-(s+1))^{\underline{\alpha-1}}(-1)^n\binom{t-s-\alpha}{n}f(s)\Delta^ng(t-\alpha-n)\\
     &=\frac{1}{\Gamma(\alpha)}\sum_{n=0}^{t-a-\alpha}\sum_{s=a}^{t-\alpha-n}(t-(s+1))^{\underline{\alpha-1}}(-1)^n\binom{t-s-\alpha}{n}f(s)\Delta^ng(t-\alpha-n).
\end{align*}
Now, using the equality $(-1)^n=\frac{(-\alpha)^{\underline{n}}}{(\alpha)_{n}}$, we get from the previous deduction and after some cancellations that,
\begin{align*}
    &\frac{1}{\Gamma(\alpha)}\sum_{n=0}^{t-a-\alpha}\sum_{s=a}^{t-\alpha-n}(t-(s+1))^{\underline{\alpha-1}}(-1)^n\binom{t-s-\alpha}{n}f(s)\Delta^ng(t-\alpha-n)\\
    &=\sum_{n=0}^{t-a-\alpha}\frac{(-\alpha)^{\underline{n}}}{n!}\frac{1}{\Gamma(n+\alpha)}\sum_{s=a}^{t-(\alpha+n)}(t-(s+1))^{\underline{\alpha+n-1}}f(s)\Delta^ng(t-\alpha-n)\\
    &=\sum_{n=0}^{t-a-\alpha}\binom{-\alpha}{n}\Delta_a^{-(\alpha+n)}f(t)\cdot\Delta^ng(t-\alpha-n),
\end{align*}
which concludes the proof.
\end{proof}

As a first consequence of Theorem \ref{thmL-R} we prove the following result.

\begin{proposition}\label{prop11111}
Consider the parameters $\alpha,\beta,\gamma$ such that $\alpha\in\mathbb{R}\backslash\mathbb{N}^0$, $\beta,\beta+\gamma\in\mathbb{R}\backslash\mathbb{N}^{-1}$. Then
\begin{multline}\label{form1}
\frac{\Gamma(\beta+\gamma+1)}{(t-(\beta+\gamma)-\alpha)!}\frac{(\beta+\gamma+\alpha+1)_{t-(\beta+\gamma)-\alpha}}{\Gamma(\beta+1)}\\
=\sum_{n=0}^{t-\alpha-\beta-\gamma}\binom{-\alpha}{n}\frac{(\alpha+\beta+n+1)_{(t-\alpha-\beta-\gamma-n)}\gamma^{\underline{n}}(t-\alpha-n)^{\underline{\gamma-n}}}{(t-\alpha-\beta-\gamma-n)!},\quad t\in\mathbb{N}_{\alpha+\beta+\gamma}.
\end{multline}
\end{proposition}

\begin{proof}
Consider the function $h(t)=t^{\underline{\beta+\gamma}}$ on $t\in\mathbb{N}_{\beta+\gamma}$. By Theorem \ref{sumPR}, we have
\begin{equation*}
    \Delta_{\beta+\gamma}^{-\alpha}h(t)=\frac{\Gamma(\beta+\gamma+1)}{(t-(\beta+\gamma)-\alpha)!}(\beta+\gamma+\alpha+1)_{t-(\beta+\gamma)-\alpha},\mbox{ for } t\in\mathbb{N}_{\alpha+\beta+\gamma}.
\end{equation*}
Now consider the functions $f,g:\mathbb{N}_{\beta+\gamma}\to\mathbb{R}$ defined by
$$f(t)=(t-\gamma)^{\underline{\beta}},\ g(t)=t^{\underline{\gamma}}.$$
Then for $t\in\mathbb{N}_{\alpha+\beta+\gamma}$ and $n\in\mathbb{N}_0^{t-(\beta+\gamma)-\alpha}$,
\begin{align*}
    \Delta_{\beta+\gamma}^{-(\alpha+n)}f(t)&=\frac{\Gamma(\beta+1)}{(t-\alpha-\beta-\gamma-n)!}(\alpha+\beta+n+1)_{(t-\alpha-\beta-\gamma-n)},\\
    \Delta^n g(t)&=\gamma^{\underline{n}}t^{\underline{\gamma-n}}.
\end{align*}
Theorem \ref{thmL-R} now implies that,
\begin{multline*}
\frac{\Gamma(\beta+\gamma+1)}{(t-(\beta+\gamma)-\alpha)!}(\beta+\gamma+\alpha+1)_{t-(\beta+\gamma)-\alpha}\\
=\sum_{n=0}^{t-\alpha-\beta-\gamma}\binom{-\alpha}{n}\frac{\Gamma(\beta+1)}{(t-\alpha-\beta-\gamma-n)!}(\alpha+\beta+n+1)_{(t-\alpha-\beta-\gamma-n)}\gamma^{\underline{n}}(t-\alpha-n)^{\underline{\gamma-n}},
\end{multline*}
for $t\in\mathbb{N}_{\alpha+\beta+\gamma}$, and \eqref{form1} follows immediately. 
\end{proof}

\begin{corollary}[Saalschutz's formula]
Suppose $a,c\in\mathbb{R}\backslash\mathbb{N}^0$ with $c-a-1\in\mathbb{R}\backslash\mathbb{N}^{-1}$, and $b\in\mathbb{R}$ with $c-a-b-1\in\mathbb{R}\backslash\mathbb{N}^{-1}$. Then for $m\in\mathbb{N}_0$
\begin{equation}\label{Salsc}
\frac{(c-a)_m(c-b)_m}{(c)_m(c-a-b)_m}={_3F}_2(a,b,-m;c,1+a+b-c-m;1).
\end{equation}
\end{corollary}

\begin{proof}
Consider the hypothesis of Proposition \ref{prop11111}, together with $\alpha+\beta\in\mathbb{R}\backslash\mathbb{N}^{-1}$. Put $a=\alpha$, $b=-\gamma$, $c=1+\alpha+\beta$ and $m=t-\alpha-\beta-\gamma$.

Then \eqref{form1} may be written as
\begin{multline*}
\frac{\Gamma(\beta+\gamma+1)}{m!}\frac{(\beta+\gamma+\alpha+1)_{m}}{\Gamma(\beta+1)}\\
=\sum_{n=0}^{t-\alpha-\beta-\gamma}\binom{-\alpha}{n}\frac{(\alpha+\beta+n+1)_{(t-\alpha-\beta-\gamma-n)}\gamma^{\underline{n}}(t-\alpha-n)^{\underline{\gamma-n}}}{(t-\alpha-\beta-\gamma-n)!}\\
=\frac{\Gamma(1+\alpha+\beta+m)}{\Gamma(m+\beta+1)}\sum_{n=0}^{m}\frac{(\alpha)_n}{n!}\frac{(-\gamma)_{n}\Gamma(m+\beta+1+\gamma-n)}{\Gamma(\alpha+\beta+n+1)(m-n)!},
\end{multline*}
where we have used $(-1)^n(\alpha)_{n}=(-\alpha)^{\underline{n}}$. Now, after some rearrangements, we get
\begin{multline*}
\frac{\Gamma(\beta+\gamma+1)}{\Gamma(1+\alpha+\beta+m)}\frac{(\beta+\gamma+\alpha+1)_{m}}{\Gamma(\beta+1)}\Gamma(m+\beta+1)\\
=\frac{\Gamma(1+\beta+\gamma+m)}{\Gamma(1+\alpha+\beta)}\sum_{n=0}^{m}\frac{(\alpha)_{n}(-\gamma)_{n}(-m)_n}{(1+\alpha+\beta)_n(-\beta-\gamma-m)_{n}}\frac{1}{n!},
\end{multline*}
or
\begin{multline*}
\frac{\Gamma(\beta+\gamma+1)\Gamma(1+\alpha+\beta)}{\Gamma(1+\alpha+\beta+m)}\frac{(\beta+\gamma+\alpha+1)_{m}\Gamma(m+\beta+1)}{\Gamma(\beta+1)\Gamma(1+\beta+\gamma+m)}\\
=\sum_{n=0}^{m}\frac{(\alpha)_{n}(-\gamma)_{n}(-m)_n}{(1+\alpha+\beta)_n(-\beta-\gamma-m)_{n}}\frac{1}{n!}.
\end{multline*}
Now, note that
\begin{align*}
    (c-a)_m&=\frac{\Gamma(1+\beta+m)}{\Gamma(1+\beta)}\\
    (c-b)_m&=\frac{\Gamma(1+\alpha+\beta+\gamma+m)}{\Gamma(1+\alpha+\beta+\gamma)}\\
     (c)_m&=\frac{\Gamma(1+\alpha+\beta+m)}{\Gamma(1+\alpha+\beta)}\\
      (c-a-b)_m&=\frac{\Gamma(1+\beta+\gamma+m)}{\Gamma(1+\beta+\gamma)}.
\end{align*}
The equality \eqref{Salsc} now follows from the arbitrariness of $\alpha,\beta,\gamma$ and $t$.
\end{proof}

\section*{Acknowledgments}
The author would like to thank the referees for their careful reading of the manuscript, and their corrections and suggestions which contributed to improve this article.

\bibliographystyle{amsplain}

\end{document}